\documentclass[a4paper, 11pt]{amsart}
\usepackage{amsmath,amssymb,amscd}
\usepackage{mathtools}
\usepackage[colorlinks=true, allcolors=blue]{hyperref}
\usepackage{url}
\usepackage[latin1]{inputenc}
\usepackage{graphicx}
\usepackage{tikz-cd}
\usepackage{etoc}
\usepackage{enumerate}
\makeatletter

\newcommand{\pn}{P_{\text{naive}}}

\newtheorem{theorem}{Theorem}[section]

\newtheorem{proposition}[theorem]{Proposition}

\newtheorem{lemma}[theorem]{Lemma}
\theoremstyle{definition}
\newtheorem{definition}[theorem]{Definition}
\newtheorem{example}[theorem]{Example}
\newtheorem{remark}[theorem]{Remark}

\usepackage{color}

\title{Property $P_{\text{naive}}$ for non-orientable big Mapping Class Groups}
\author{Jes\'us HERN\'ANDEZ HERN\'ANDEZ and Tianyi LOU}
\date{\today}
\begin{document}
\maketitle

\begin{abstract}
We give a criterion for property \(P_{\text{naive}}\) for subgroups of orientable big mapping class groups and apply it to mapping class groups of non-orientable infinite type surfaces through the orientation double cover.  
\end{abstract}

\section{Introduction}

Let \(G\) be a group. Recall that \(G\) has property
\(P_{\text{naive}}\) if, for every finite subset
\(F \subset G \setminus \{1\}\), there exists an element \(g \in G\) of
infinite order such that, for every \(h \in F\),
\[
\langle h,g\rangle \cong \langle h\rangle * \langle g\rangle .
\]
This property was introduced by Bekka--Cowling--de la Harpe \cite{BCD95} in connection
with the simplicity of reduced \(C^*\)-algebras.

Later on, Abbott and Dahmani \cite{AD18} proved that every acylindrically hyperbolic group
with no non-trivial finite normal subgroup has property \(P_{\text{naive}}\).  In
particular, this applies to the mapping class groups of sufficiently complex
orientable finite-type surfaces after the usual exceptional cases are
excluded.  The situation is different for infinite-type surfaces. Their
mapping class groups are not acylindrically hyperbolic in general
\cite{BG18}. Nevertheless, in the orientable case, if the surface contains a
non-displaceable finite-type subsurface, the projection-complex machinery of
Bestvina-Bromberg-Fujiwara \cite{BBF15} and the work of
Horbez-Qing-Rafi \cite{HQR20} provide a useful hyperbolic action.  This
makes it possible to prove property \(P_{\text{naive}}\) by a combination of finite-type arguments and ping-pong; see the second author's work \cite{TL24}.

The purpose of this paper is to develop a corresponding argument for
non-orientable big mapping class groups.  The natural tool is the orientation
double cover
\[
        p:\widetilde N\longrightarrow N.
\]
The surface $\widetilde N$ is orientable, so the orientable theory is
available upstairs.  There is, however, an important obstruction to a naive
descent argument: property \(P_{\text{naive}}\) for $\operatorname{Map}(\widetilde N)$ does not by
itself provide a witness which descends to $N$.

The correct formulation is subgroup-theoretic.  The canonical lift embeds
$\operatorname{Map}(N)$ in $\operatorname{Map}(\widetilde N)$ by work of Colin-Hidalgo-Jim\'enez Rolland-Morales-Quispe \cite{CHR26}, and the orientable ping-pong construction
must be run directly inside this lifted subgroup.  A second issue is that it
is not enough to require a prescribed element $h$ to act non-trivially on a
chosen finite-type subsurface: a non-trivial power of $h$ may return to the
subsurface and act trivially there.  The relevant condition is instead
\(
\langle h\rangle\cap \operatorname{Fix}(K)=\{1\}.
\)
This controls every non-trivial return power simultaneously.

We therefore first establish an orientable subgroup criterion in which a
finite-type quotient of the stabilizer is required to be acylindrically
hyperbolic without non-trivial finite normal subgroup.  This finite-type hypothesis is the
precise input needed to apply Abbott-Dahmani \cite{AD18} and to choose one common
pseudo-Anosov witness for all return subgroups.

For the non-orientable application, write
\[
        G:=\iota(\operatorname{Map}(N))<\operatorname{Map}(\widetilde N)
\]
for the lifted subgroup.  If $K\subset N$ is connected and non-orientable,
then $\widetilde K=p^{-1}(K)$ is connected.  The main theorem is stated using
the following explicit finite-type condition.

Let $N$ be a connected non-orientable infinite-type surface and let
$p:\widetilde N\to N$ be its orientation double cover. We say that $N$
satisfies the \emph{symmetric finite-type quotient condition} if there is a
connected non-orientable finite-type non-displaceable subsurface $K_0\subset N$
with the following cofinality property.  For every connected essential
non-orientable finite-type subsurface $K_1\supset K_0$, there exists a
connected essential non-orientable finite-type subsurface $K\supset K_1$
such that
\[
 Q_{G}(\widetilde K)
 :=\rho_{\widetilde K}\bigl(\operatorname{Stab}_{G}(\widetilde K)\bigr)
 \leq \operatorname{Map}(\widehat{\widetilde K})
\]
acts non-elementarily and acylindrically on $\mathcal C(\widehat{\widetilde K})$ and has no
non-trivial finite normal subgroup.  Here
$\widetilde K=p^{-1}(K)$ and $\widehat{\widetilde K}$ is obtained from
$\widetilde K$ by gluing a once-punctured disk to every boundary component, and $\rho_{\widetilde K}: \operatorname{Stab}_{G}(\widetilde K)\rightarrow \operatorname{Map}(\widehat{\widetilde{K}})$ is the homomorphism induced by the restriction and capping homomorphisms (see Section \ref{subsec:Prelim Non-disp}). This symmetric finite-type quotient condition also applies to the orientable case.

\begin{theorem}\label{thm:main}
Let $N$ be a connected non-orientable surface of infinite type that admits a non-displaceable finite-type subsurface. Then $\operatorname{Map}(N)$ has property
%Let $N$ be a connected non-orientable surface of infinite type that satisfies the symmetric finite-type quotient condition. Then $\operatorname{Map}(N)$ has property
\(P_{\text{naive}}\).
\end{theorem}

The proof has two independent parts.  In Section~\ref{sec:subgroup} we prove the
orientable subgroup criterion.  In Section~\ref{sec:nonorientable} we prove that
all powers of a finite set can be detected in a $\tau$-invariant finite-type
subsurface of the orientation double cover, then we also prove that the surface satisfies the symmetric finite-type quotient condition, and finally we apply the subgroup
criterion to $G$.  Since the resulting witness already lies in $G$, no
additional symmetrization is required.

\subsection*{Acknowledgements}
The first author would like to thank Rita Jim\'enez Rolland and N\'estor Colin for enlightening conversations. Also, the first author was partially funded during the creation of this work by the UNAM PAPIIT research grants IN103026 and IN106925. The second author would like to thank Indira Chatterji and Fran\c{c}ois Dahmani for their guidance and for many illuminating discussions. 

% ------------------------------------------------------------
\section{Background and notation}
% ------------------------------------------------------------

\subsection{Non-orientable surfaces}

Let \(N\) be a connected, second-countable, Hausdorff surface. In this
paper, \(N\) is non-orientable. Unless otherwise stated, $N$ has empty boundary.

For a non-orientable finite type surface, we write $N_{g,b,p}$ for the connected non-orientable surface of genus \(g\), with \(b\) boundary
components and \(p\) punctures. Recall that for a non-orientable finite-type surface $N$, its genus is the number of projective planes whose connected sum forms $N$ along with the punctures and boundary components. For an infinite-type surface it is the supremum of the genera of its finite-type subsurfaces.
%a genus \(g\) non-orientable
%surface is the connected sum of \(g\) projective planes. 

If \(N\) is non-orientable, we define
\[
\operatorname{Map}(N)=\operatorname{Homeo}(N)/\operatorname{Homeo}_{0}(N),
\]
where \(\operatorname{Homeo}_{0}(N)\) denotes the connected component of the identity in
\(\operatorname{Homeo}(N)\), equipped with the compact-open topology. If the surface were orientable, one usually takes orientation-preserving
homeomorphisms. For the present paper, the non-orientable case is the
main object.

\subsection{Curves on non-orientable surfaces}

Let \(N\) be a compact non-orientable surface. A simple closed curve
\(\alpha\subset N\) is called:
\begin{itemize}
    \item one-sided if a regular neighborhood of \(\alpha\) is a
    M\"{o}bius band;
    \item two-sided if a regular neighborhood of \(\alpha\) is an
    annulus.
\end{itemize}

A curve is called \emph{essential} if it does not bound a disk, does not bound a once-punctured disk, does not bound a M\"obius band, and is not isotopic to a boundary component.

In the non-orientable setting, one must be careful: Dehn twists only exist
around two-sided curves. There is no Dehn twist around a one-sided curve.
This is one of the main differences from the orientable case.

The curve graph \(\mathcal{C}(N)\) has vertices given by isotopy classes of essential
simple closed curves, and two vertices are joined by an edge if they admit
disjoint representatives.

Kuno \cite{Kuno16} proves that if \(\mathcal{C}(N)\) is connected, then it is uniformly
hyperbolic in the
non-orientable setting. In fact, it is \(17\)-hyperbolic.

\subsection{Non-displaceable subsurfaces}\label{subsec:Prelim Non-disp}

For this subsection let $\Sigma$ be a surface, possibly orientable or non-orientable.

\begin{definition}
%Let \(\Sigma\) be a surface. 
A connected finite-type subsurface \(K\subset \Sigma\) is
called \emph{non-displaceable} if
$f(K)\cap K\neq \varnothing$
for every \(f\in \operatorname{Homeo}(S)\).
\end{definition}

Note that this definition does not depend on the orientability of the surface.
%Hence, it makes sense for non-orientable surfaces as well.

\begin{remark}
If \(K'\subset \Sigma\) is non-displaceable and \(K\subset \Sigma\) is a connected
subsurface containing \(K'\), then \(K\) is also non-displaceable.
\end{remark}

We shall need the following relative version of non-displaceability.

\begin{definition}
Let \(\Sigma\) be a surface, let \(H\leq \operatorname{Map}(\Sigma)\), and let
\(K\subset \Sigma\) be a connected finite type subsurface. We say that \(K\)
is \emph{non-displaceable relative to \(H\)} if
\(
\widetilde{h}(K)\cap K\neq\varnothing
\)
for every \(\widetilde{h}\in \pi^{-1}(H)\), where $\pi: \operatorname{Homeo}(\Sigma) \to \operatorname{Map}(\Sigma)$ is the natural projection.
\end{definition}

Henceforth, we might abuse notation and language and say $h(K) \cap K$ for $h \in \operatorname{Map}(\Sigma)$ when we mean $\widetilde{h}(K)\cap K$ for every \(\widetilde{h}\in \pi^{-1}(H)\).

Let $K\subset \Sigma$ be a connected essential finite-type subsurface. Let $\operatorname{Stab}_{\operatorname{Map}(\Sigma)}(K)$ be the subgroup of $\operatorname{Map}(\Sigma)$ made of all mapping classes that preserve the isotopy class of $K$. Define
\(
\operatorname{Stab}_H(K):=H\cap\operatorname{Stab}_{\operatorname{Map}(\Sigma)}(K).
\)
Let $\operatorname{Fix}_{\operatorname{Map}(\Sigma)}(K)$ be the subgroup of $\operatorname{Map}(\Sigma)$ made of all elements that have a representative $\psi \in \operatorname{Homeo}(\Sigma)$ such that $\psi(K)=K$ and $\psi|_{K}=\mathrm{Id}_{K}$. Define
\(
\operatorname{Fix}_H(K):=
 H\cap\operatorname{Fix}_{\operatorname{Map}(\Sigma)}(K).
\)

Let $\widehat K$ be obtained from $K$ by gluing a once-punctured disk to
every boundary component.  As in \cite[Lemma 2.4]{HQR20}, restriction to
$K$ followed by capping gives a homomorphism
\[
\rho_K:\operatorname{Stab}_{\operatorname{Map}(\Sigma)}(K)\longrightarrow \operatorname{Map}(\widehat K)
\]
with
\(\ker\rho_K=\operatorname{Fix}_{\operatorname{Map}(\Sigma)}(K).
\)
We put
\[
Q_H(K):=\rho_K(\operatorname{Stab}_H(K))\leq\operatorname{Map}(\widehat K).
\]
Thus
\(\ker\bigl(\rho_K|_{\operatorname{Stab}_H(K)}\bigr)=\operatorname{Fix}_H(K).
\)

% ------------------------------------------------------------
\section{A subgroup criterion for \texorpdfstring{$P_{\text{naive}}$}{P\_naive}}
\label{sec:subgroup}

The following examples illustrate that property \(P_{\text{naive}}\) is not
inherited by arbitrary subgroups. Thus, in order to prove that a subgroup
has \(P_{\text{naive}}\), one has to check that the ping-pong argument can be
performed inside the subgroup itself.

\begin{example}
Let $G=\mathbb{Z}^{2}*\mathbb{Z}$. Then \(G\) is acylindrically hyperbolic and has no non-trivial finite normal
subgroups. Hence, by \cite{AD18}, \(G\) has property \(P_{\text{naive}}\). However, the subgroup $\mathbb{Z}^{2}<G$ does not have property \(P_{\text{naive}}\) because \(\mathbb{Z}^{2}\) is abelian.
\end{example}

There are many similar examples. For instance, the free group \(F_{2}\) has property \(P_{\text{naive}}\). However, it contains an infinite cyclic subgroup $\langle a\rangle \cong \mathbb{Z}$. This subgroup \(\langle a\rangle\) does not have Property \(P_{\text{naive}}\), since it is abelian.

These examples show that property \(P_{\text{naive}}\) is not,
in general, inherited by arbitrary subgroups. Therefore, if $S$ is an orientable
surface, in order to prove
that a subgroup \(H<\operatorname{Map}(S)\) has \(P_{\text{naive}}\) one should
not argue by inheritance from \(\operatorname{Map}(S)\). Instead, one has to
check that the ping-pong construction can be performed directly inside
\(H\). 

The following lemma provides a finite-type subsurface on which the
given elements have no non-trivial powers acting trivially. 

\begin{lemma}\label{power-detect}
Let $S$ be a connected orientable surface of infinite type and let
$H<\operatorname{Map}(S)$. Suppose that $K_0\subset S$ is a connected finite-type
subsurface which is non-displaceable relative to $H$. For every finite subset
\(
F=\{h_1,\ldots,h_n\}\subset H\setminus\{1\},
\)
there exists a connected essential finite-type subsurface
$K\supset K_0$, still non-displaceable relative to $H$, such that
\[
\langle h_i\rangle\cap\operatorname{Fix}_H(K)=\{1\}
\qquad (1\leq i\leq n).
\]
\end{lemma}

\begin{proof}
Fix $i$. Suppose first that $h_i$ has infinite order. By
\cite[Lemma 12]{ACCL21}, there exists an essential curve $c_i$ whose
$\langle h_i\rangle$-orbit is infinite. Hence
\(
h_i^m(c_i)\neq c_i\) for every \(m\neq 0.
\)
If $h_i$ has finite order $r_i$, then, for every
$1\leq m<r_i$, faithfulness of the action of $\operatorname{Map}(S)$ on the
curve graph gives an essential curve $c_{i,m}$ such that
\(
h_i^m(c_{i,m})\neq c_{i,m}.
\)

Choose a connected finite-type subsurface containing $K_0$ and all
the finitely many curves chosen above, and enlarge it, if necessary,
so that it is essential.
Since it contains $K_0$, it remains non-displaceable relative to $H$.

If $h_i^m\in\operatorname{Fix}_H(K)$, then $h_i^m$ fixes every chosen curve
contained in $K$. The preceding construction therefore implies
$h_i^m=1$. Hence
$\langle h_i\rangle\cap\operatorname{Fix}_H(K)=\{1\}.$
\end{proof}

With this lemma in mind, henceforth we always assume that if $\{h_{1}, \ldots, h_{n}\} \subset H \setminus \{1\}$ and $K$ is a non-displaceable subsurface of finite type relative to $H$, then $\langle h_i \rangle \cap \operatorname{Fix}_H(K) = \{1\}$.

\subsection{The relative projection-complex action}
Let $S$ be an orientable surface of infinite type, and $K \subset S$ be a connected, essential, finite-type subsurface non-displaceable relative to
$H$. Put
\[
       \mathbf{Y}_{H,K}:=H\cdot[K].
\]
For $Y\in\mathbf{Y}_{H,K}$ let $\mathcal{C}_S(Y)$ denote the subgraph of $\mathcal{C}(S)$ induced by the curves that can be represented essentially in $Y$, and recall that $\mathcal{C}_S(Y) \cong \mathcal{C}(Y)$ (see for example \cite{HQR20}).
%curve graph piece consisting
%of curves represented essentially in $Y$. 
Since $K$ is non-displaceable relative to $H$, any
two distinct members of $\mathbf{Y}_{H,K}$ overlap essentially, i.e. for any $[K_1], [K_{2}] \in \mathbf{Y}_{H,K}$ there are no isotopic representatives of $K_1$ and $K_2$ that are disjoint.  Indeed, relative
non-displaceability gives non-empty intersection; if their boundaries did not
meet essentially, two homeomorphic essential representatives of the same
finite-type subsurface would be nested, and the absence of disk,
once-punctured-disk and annular complementary redundancies would force them
to be isotopic.  Hence subsurface projections between distinct elements of
$\mathbf{Y}_{H,K}$ are defined exactly as in \cite{BBF15,HQR20}.

Note however that for subsurface projections to be defined we need that $|\mathbf{Y}_{H,K}| > 1$. With this assumption, the proofs of the projection axioms in \cite{BBF15,HQR20} use only this
pairwise-overlap projection data, and therefore apply verbatim to the
$H$-orbit $\mathbf{Y}_{H,K}$.  Therefore the projection construction gives,
for a sufficiently large projection parameter, a hyperbolic total space
\[
       \mathbb{X}_H(K):=\mathcal{C}(\mathbf{Y}_{H,K})
\]
with an isometric action of $H$ where each piece $\mathcal{C}_S(Y)$ is isometrically
embedded.

\begin{remark}\label{rmk:YHK1}
    If $|\mathbf{Y}_{H,K}| = 1$, we have that $H \leq \operatorname{Stab}_{\operatorname{Map}(S)}(K)$. This implies that $\mathbb{X}_H(K) \cong \mathcal{C}(K)$ and it admits an isometric $H$-action given by $\rho_K$, where $\operatorname{Fix}_H(K)$ acts trivially.
\end{remark}

Recall that an element \(f \in \operatorname{Map}(S)\) is \emph{\(K\)-pseudo-Anosov} if \(f\) preserves the isotopy class of the subsurface $K \subset S$ and, denoting as before by \(\widehat{K}\) a surface obtained from \(K\) by gluing a once-punctured disk on every boundary component of \(K\), the mapping class \(f\) induces a pseudo-Anosov mapping class of \(\widehat{K}\).

Also recall from Handel-Mosher \cite{HM21} that if $X$ is a hyperbolic metric space, $G$ is a group acting by isometries on $X$ and $g \in G$ acts as a loxodromic isometry on $X$ with fixed points $g^+,g^- \in \partial X$, we say that $g$ is WWPD if for every $x \in X$ and $R > 0$, there exists an integer $M \geq 1$ such that any subset $S \subset G$ that satisfies the following properties is finite:
\begin{enumerate}
    \item For each $h \in S$ we have $d(x, h(x)) < R$ and $d(g^{M}(x), g^{M}(h(x)) < R$.
    \item For each $s_{1}, s_{2} \in S$ with $s_{1} \neq s_{2}$, $s_{1}$ and $s_{2}$ lie in different left cosets of $\operatorname{Stab_{G}(g^+,g^-)}$.
\end{enumerate}

Moreover, from Bestvina-Fujiwara \cite{BF02} we say the $g$ is WPD if for every $x \in X$ and $R > 0$, there exists an integer $M \geq 1$ such that the set from (1) above taking $S = G$ is finite.

\begin{proposition}\label{prop:relativeHQR}
Let $K \subset S$ be connected, essential, finite-type and non-displaceable subsurface relative to
$H$. Then $H$ acts isometrically on the hyperbolic space
$\mathbb{X}_H(K)$.  If $g\in H$ is $K$-pseudo-Anosov, equivalently if
$g\in\operatorname{Stab}_H(K)$ and $\rho_K(g)$ is pseudo-Anosov on $\widehat K$, then
$g$ is a WWPD loxodromic element for the $H$-action on $\mathbb{X}_H(K)$.
\end{proposition}

\begin{proof}
If $|\mathbf{Y}_{H,K}| = 1$, as mentioned in Remark \ref{rmk:YHK1}, $H$ acts by isometries in $\mathbb{X}_H(K) \cong \mathcal{C}(K)$ via $\widehat{H}:=\rho_K(H)$. Suppose that $\rho_K(g)$ is a pseudo-Anosov on $\widehat K$, then by the work of Masur-Minsky \cite{MM99} and Bestvina-Fujiwara \cite{BF02} we have that both $g$ and $\rho_K(g)$ acts as loxodromic isometries on $\mathcal{C}(K)$ and $\rho_K(g)$ acts WPD on $\mathcal{C}(K)$. Now, due to the definition of the action we have the following: Let $x \in \mathcal{C}(K)$ and $R > 0$.
\begin{itemize}
    \item If $T \subset H$ satisfies (1) from the definition of WWPD loxodromic isometries for some $M \geq 1$, then $\rho_K(T) \subset \widehat{H}$ also satisfies it.
    \item If $T \subset H$ has infinitely many representatives of left cosets of $\operatorname{Stab}_{H}(g^+,g^-)$, then $\rho_K(T)$ also has infinitely many representatives.
\end{itemize}
If $T \subset H$ is a set that satisfies (1) and (2) from the definition of WWPD loxodromic isometries, then by the argument above, the fact that $\rho_K(g)$ is WPD implies that $\rho_K(T)$ is finite, and thus $\rho_K(T)$ has finitely many representatives of left cosets of $\operatorname{Stab}_{H}(g^+,g^-)$. Hence $T$ has finitely many representatives of left cosets of $\operatorname{Stab}_{H}(g^+,g^-)$. Now, (2) implies that $T$ has to be finite. Therefore, $g$ is WWPD.

Suppose now that $|\mathbf{Y}_{H,K}| > 1$. The projection axioms and the
construction of the total space depend only
on the projection data, and hence remain valid after restricting the full
orbit to the $H$-orbit.  A $K$-pseudo-Anosov acts loxodromically on the
isometrically embedded piece $\mathcal{C}_S(K)\cong\mathcal{C}(\widehat K)$.  The WWPD
argument of \cite[Theorem 2.9]{HQR20} uses the same projection bounds and
therefore applies to the restricted orbit.  Equivalently, the bounded
projection characterization of WWPD is inherited by the restricted
projection system.
\end{proof}

Bowditch proved that the action of
\(\operatorname{Map}(\widehat K)\) on
\(\mathcal C(\widehat K)\cong\mathcal C_S(K)\) is acylindrical
\cite{B08}. By enlarging $K$ if necessary, we may assume that $\widehat K$ is sufficiently complex so that $\operatorname{Map}(\widehat K)$ is acylindrically hyperbolic and has no non-trivial finite normal subgroup. This can be arranged since $S$ is of infinite type. Therefore, $\operatorname{Map}(\widehat K)$ has property $P_{\mathrm{naive}}$ by \cite{AD18}.

For every finite subset
\(
F=\{h_1,\ldots,h_n\}\subset H\setminus\{1\}
\), the following lemma shows that the subgroup $H$ has property $P_{\text{naive}}$ when all $h_{i}$ preserving $[K]$.

\begin{lemma}\label{lem:common-witness}
Let $S$ be an infinite-type surface, $H \leq \operatorname{Map}(S)$ and
$F = {h_{1}, \ldots , h_{n}} \subset H \setminus \{1\}$. Suppose that
$K\subset S$ is a non-displaceable finite-type subsurface relative to $H$ and satisfy the
symmetric finite-type quotient condition. Then there exists a $K$-pseudo-Anosov
element $g\in H$ such that, for every $i$ for which $h_i$ preserves the
isotopy class of $K$, if
\(
\langle h_i\rangle\cap\operatorname{Fix}_H(K)=\{1\},
\)
then
\(
\langle h_i,g\rangle
\cong
\langle h_i\rangle*\langle g\rangle.
\)
\end{lemma}

\begin{proof}
Let
\(
I=\{i\in\{1,\ldots,n\}:h_i\text{ preserves the isotopy class of }K\}.
\)
For every $i\in I$, we have $h_i\in\operatorname{Stab}_H(K)$ and hence
$\rho_K(h_i)$ is defined. Moreover,
\(
\langle h_i\rangle\cap
\ker\bigl(\rho_K|_{\operatorname{Stab}_H(K)}\bigr)
=\{1\},
\)
so the restriction of $\rho_K$ to $\langle h_i\rangle$ is injective.
In particular, $\rho_K(h_i)\neq1$ for every $i\in I$.

Apply the proof of Abbott-Dahmani~\cite[Theorem~2.3]{AD18} to the action of $Q_H(K)$ on $\mathcal C(\widehat K)$ and to the elements
\(
\rho_K(h_i), i\in I.
\)
Their proof first produces, by finitely many cone-offs, a hyperbolic
space $Y$ on which all the elements $\rho_K(h_i)$ are elliptic. It then
produces a loxodromic element $\gamma\in Q_H(K)$ such that, for some
sufficiently large $N$, the element
\(
\overline g=\gamma^N
\)
satisfies
\(
\langle\rho_K(h_i),\overline g\rangle
\cong
\langle\rho_K(h_i)\rangle*
\langle\overline g\rangle
\)
for every $i\in I$.

As in the proof of Abbott-Dahmani, the successive cone-offs give an
equivariant Lipschitz map
\(
q:\mathcal C(\widehat K)\rightarrow Y.
\)
Thus, for some $L>0$,
\[
d_Y(q(x),\overline g^m q(x))
\leq
Ld_{\mathcal C(\widehat K)}
(x,\overline g^m x).
\]
Since $\overline g$ is loxodromic on $Y$, the left-hand side grows
linearly in $|m|$. Hence the orbit of $x$ under $\overline g$ also grows
linearly in $\mathcal C(\widehat K)$, so $\overline g$ is loxodromic on
$\mathcal C(\widehat K)$. By the Nielsen-Thurston classification,
$\overline g$ is pseudo-Anosov.

Since $\overline g\in Q_H(K)$ by construction, there exists
\(
g\in\operatorname{Stab}_H(K)\subset H
\)
such that
\(
\rho_K(g)=\overline g.
\)
In particular, $g$ is $K$-pseudo-Anosov.

Fix $i\in I$, and let $w$ be a non-trivial reduced word in
$\langle h_i\rangle*\langle g\rangle$. Every non-trivial syllable
belonging to $\langle h_i\rangle$ has non-trivial image under $\rho_K$,
because $\rho_K$ is injective on $\langle h_i\rangle$. Similarly, every
non-trivial syllable belonging to $\langle g\rangle$ has non-trivial
image, since
\(
\rho_K(g)=\overline g
\)
has infinite order. Consequently, $\rho_K(w)$ is a non-empty reduced
word in
\(
\langle\rho_K(h_i)\rangle*
\langle\overline g\rangle,
\)
and hence $\rho_K(w)\neq1$. It follows that $w\neq1$.

Thus the natural homomorphism
\(
\langle h_i\rangle*\langle g\rangle
\rightarrow
\langle h_i,g\rangle
\)
is injective. Since it is clearly surjective, it is an isomorphism.
This holds for every $i\in I$.
\end{proof}

Recall that an \emph{ending lamination} on a finite type surface is a geodesic lamination $\Lambda$ which is filling (every essential simple closed curve meets $\Lambda$, equivalently $S\setminus \Lambda$ is a union of disks and
once-punctured disks) and minimal, meaning that every leaf of $\Lambda$ is dense in $\Lambda$ (equivalently, $\Lambda$ has no proper nonempty closed
sublamination).
Observe that if $\Lambda$ is an ending lamination on $K_{1}$ and on $K_{2}$ (viewed as subsurfaces of the same ambient surface), then $K_{1}=K_{2}$ and both coincide with the support of $\Lambda$.
We will need the following lemma about elements that do not preserve \([K]\)
and \(K\)-pseudo-Anosov elements of \(\operatorname{Map}(S)\).
\begin{lemma}\label{1}
    Let $S$ be a connected orientable surface of infinite type, $H \leq \operatorname{Map}(S)$ and $\{h_{1}, \ldots, h_{n}\} \subset H \setminus \{1\}$. Suppose that $K \subset S$ is a non-displaceable finite-type subsurface relative to $H$ and that $g$ is the K-pseudo-Anosov element supplied by Lemma \ref{lem:common-witness}. Let \(\mathbb{X}_{H}\) be the hyperbolic space associated with $K$ on which \(H\) admits a continuous non-elementary isometric action, as in Proposition \ref{prop:relativeHQR}, where $g$ acts as a WWPD loxodromic with fixed points $g^+, g^- \in \partial \mathbb{X}_{H}$.
    %Let $S$ be a connected orientable surface of infinite type and $K \subset S$ be a non-displaceable essential subsurface of finite type. Let \(\mathbb{X}_{H}\) be a hyperbolic space associated with $K$ on which \(H\) admits a continuous nonelementary isometric action. Let $h_1,\ldots,h_n\in H\setminus\{1\}$ and $g$ be the K-pseudo-Anosov element supplied by Lemma \ref{lem:common-witness}. 
    
    Suppose also that each $h_{i}$ %$h_i \in \operatorname{Map}(S)$ 
    does not preserve the isotopy class of $K$. Then for every \(i\in\{1,\ldots,n\}\) and every \(m\in\mathbb Z\)
    such that \(h_i^m\neq1\),
    \begin{equation}\label{eq:endpoint-separation}
        h_i^m\{g^-,g^+\}\cap\{g^-,g^+\}
        =\varnothing.
    \end{equation}
\end{lemma}

\begin{proof} For each $i$, set
$R_i=\langle h_i\rangle\cap\operatorname{Stab}_{H}(K)$.
If $R_i\neq\{1\}$, choose a generator $s_i$ of $R_i$. By Lemma \ref{power-detect},
$\langle s_i\rangle\cap\operatorname{Fix}_{H}(K)=\{1\}$.
Fix \(i\) and
\(m\in\mathbb Z\) such that
\(
k:=h_i^m\neq1.
\)
There are two cases.

Suppose first that $k$ does not preserve the isotopy class of $K$.
By the Bestvina-Bromberg-Fujiwara construction, the pieces
$\mathcal{C}_S(K)$ and $\mathcal{C}_S(k\cdot K)$ are totally geodesically embedded in $\mathbb{X}_{H}$, and the
nearest-point projection between two distinct pieces has uniformly bounded
diameter \cite[Theorem A]{BBF15}. In particular, distinct pieces have disjoint
boundary images in $\mathbb X$.

Masur and Minsky \cite{MM99} proved that
\(\mathcal C_S(K)\) is hyperbolic, and Klarreich \cite{K22}
(see also Hamenst\"{a}dt \cite{HU06}) identified its boundary with the space of ending
laminations on \(K\). The points \(g^+\) and \(g^-\) therefore
correspond to the stable and unstable ending laminations
\(\Lambda^+\) and \(\Lambda^-\) of \(g\), both of which have support
\(K\).

If \(k(g^+)=g^+\), then \(k(\Lambda^+)=\Lambda^+\). Comparing supports
gives
\(
k \cdot K=K,
\)
contrary to the assumption. The same argument rules out
\(k(g^-)=g^-\).
If \(k(g^+)=g^-\), then \(k\) sends an ending lamination supported on
\(K\) to another ending lamination supported on \(K\). Again, equality
of supports gives \(k \cdot K=K\), a contradiction. The case
\(k(g^-)=g^+\) is identical. Hence
\[
k \cdot \{g^-,g^+\}\cap\{g^-,g^+\}=\varnothing.
\]

Suppose now that $k$ preserves $K$. Then $k\in R_i\setminus\{1\}$, so
$k=s_i^\ell$ for some $\ell\in\mathbb Z\setminus\{0\}$. Put
$a_i=\rho_K(s_i)$ and $x=\rho_K(k)=a_i^\ell$. Since
$\ker\rho_K=\operatorname{Fix}_{H}(K)$ by
\cite[Lemma 2.4]{HQR20} and
$\langle h_i\rangle\cap\operatorname{Fix}_{H}(K)=\{1\}$,
we have $x\neq1$. Write $\overline g=\rho_K(g)$.

If $x$ carried a fixed point of $\overline g$ in $\partial \mathbb{X}_H$ to a fixed point of $\overline g$ in $\partial \mathbb{X}_H$, then 
%either $\overline g^+$ or $\overline g^-$ to either $\overline g^+$ or $\overline g^-$, then
%one endpoint of $\overline g$ to an endpoint of $\overline g$, then
$\overline g$ and $x\overline g x^{-1}$ would share a boundary fixed point.
Since the action of $\operatorname{Map}(\widehat K)$ on
$\mathcal{C}(\widehat K)$ is acylindrical, these two loxodromic elements have the same
pair of fixed points. Hence $x$ preserves the unordered pair
$\{\overline g^-,\overline g^+\}$.
Since $\overline  g$ is WPD in $\mathcal{C}(\widehat K)$, hence WWPD in $\mathbb{X}_H$, Proposition and Definition~2.3(5) of
\cite{HM21} implies that these two fixed-point pairs are equal.

Thus $x\in E(\overline g)$. By \cite[Corollary 6.6]{DGO17}, there exists
$p\geq1$ such that $x^{-1}\overline g^p x=\overline g^{\pm p}$.
But $x=a_i^\ell\neq1$, so
$x^{-1}\overline g^p x\overline g^{\mp p}$ is both trivial and a non-empty reduced word in
$\langle a_i\rangle*\langle\bar g\rangle$, contradicting the choice of
$g$ in Lemma \ref{lem:common-witness}. Hence $k$ cannot carry a fixed point of $\overline g$ in $\partial \mathbb{X}_H$ to a fixed point of $\overline g$ in $\partial \mathbb{X}_H$.%either endpoint of $g$ to either endpoint of $g$.
\end{proof}
This motivates the following subgroup criterion.

\begin{proposition}[Subgroup criterion]\label{prop:subgroup-criterion}
Let $S$ be a connected orientable surface of infinite type and let
$H<\operatorname{Map}(S)$.   Assume that, for every finite
\(F\subset H\setminus\{1\}\), there exists a connected essential finite-type
\(K\subset S\) such that:
\begin{enumerate}
\item \(K\) is non-displaceable relative to \(H\);
\item \(\langle h\rangle\cap\operatorname{Fix}_H(K)=\{1\}\) for every \(h\in F\);
\item \(Q_H(K)\) acts non-elementarily on
      \(\mathcal C(\widehat K)\) by isometries, and has no nontrivial finite normal
      subgroup.
\end{enumerate}
Then \(H\) has property \(\pn\).
\end{proposition}
\begin{proof}
Let $F=\{h_1,\ldots,h_n\}$.  Choose \(K\) satisfying assumptions \((1)-(3)\) for \(F\), let $\mathbb{X}_H$ be the hyperbolic space associated to $H$ and $K$.  For each \( h_i \in F \), where \( 1 \leq i \leq n \), we analyze the element \( h_i \) via the isometric action of \( H \) on \( \mathbb{X}_H \). 

After re-indexing, we may assume that $h_{1}, \cdots, h_{k}$ preserve $[K]$, $0 \leq k \leq n$, and $h_{k+1}, \cdots, h_{n}$ do not preserve $[K]$. For each $i$, let
$R_i=\langle h_i\rangle\cap
\operatorname{Stab}_{H}(K)$.
For every non-trivial $R_i$, choose a generator $s_i$. By Lemma~\ref{lem:common-witness}, applied simultaneously to all the non-trivial generators
$s_i$, we may choose a $K$-pseudo-Anosov element $g$ as in
Lemma~\ref{1}. In particular, $g$ works for $h_1,\ldots,h_k$, since
$R_i=\langle h_i\rangle$ whenever $h_i$ preserves $K$.
Moreover, every positive power of $g$ still works for these elements.

 By \cite[Subsections 4.2,4.3,4.4]{TL24}, a sufficiently large positive power $g$ also work for $h_{k+1}, \cdots, h_{n}$. Hence, after replacing $g$ by a sufficiently large positive power,
the resulting $K$-pseudo-Anosov element works for all $h_i$.
\end{proof}

\subsection{Examples of the subgroup criterion}

For the examples of subgroups that satisfy the subgroup criterion (and thus, have property $P_{\mathrm{naive}}$), recall the following: 
\begin{itemize}
    \item $\operatorname{PMap}(S)$ is the subgroup of the mapping class group that acts trivially on the space of ends of $S$.
    \item If $S$ is a finite-type surface with at least $2g +3$ punctures, then $\operatorname{PMap}(S)$ is torsion-free. If $g = 0$, it follows from a basic argument using consecutive Birman short exact sequences for pure mapping class groups and the fact that the pure mapping class group of a sphere minus three points is trivial. If $g \geq 1$, using a argument based on Nielsen realization and the Riemann-Hurwitz formula we have that $\operatorname{PMap}(S)$ is torsion-free; see Lu \cite[Lemma 1.1 (i)]{Lu02}. 
    \item $\operatorname{PMap}_c(S)$ is the subgroup of the mapping class group whose elements admit representatives with compact support.
    \item $\operatorname{PMap}_c(S)$ is a normal subgroup of $\operatorname{PMap}(S)$, $\operatorname{PMap}_c(S)$ is generated by all the Dehn twists along essential simple closed curves on $S$, and $\operatorname{PMap}_c(S)$ is dense in $\operatorname{PMap}(S)$ if and only if $S$ has at most one end accumulated by genus (handle-shifts between two different ends accumulated by genus cannot be approximated by elements in $\operatorname{PMap}_c(S)$). See Patel-Vlamis \cite{PV18} and Aramayona-Patel-Vlamis \cite{APV20}.
\end{itemize}

\begin{example}[Pure mapping class groups]
\label{ex:pure-map}
Let \(S\) be a connected orientable infinite-type surface containing a
connected finite-type non-displaceable subsurface. Then
\(\operatorname{PMap}(S)\) has property \(P_{\mathrm{naive}}\). Indeed, every subsurface which is non-displaceable for
\(\operatorname{Map}(S)\) is also non-displaceable relative to
\(\operatorname{PMap}(S)\). Moreover, \[\operatorname{PMap}_{c}(S)\leq \operatorname{PMap}(S),\] so every sufficiently large finite-type subsurface supports compactly supported pseudo-Anosov mapping classes and Dehn twists belonging to \(\operatorname{PMap}(S)\). Hence the corresponding finite-type quotient (which is isomorphic to $\operatorname{PMap}(\widehat{K})$ for some sufficiently complex subsurface $K$; see Figure \ref{fig:fig1} for an example) acts non-elementarily in a hyperbolic space $\mathcal{C}(\widehat{K})$. To see that it has no non-trivial finite normal subgroup recall that $\operatorname{PMap}(\widehat{K})$ is torsion-free for a sufficiently complex $K$. %\jh{The finite-type group that is acting non-elementarily is the pure mapping class group of the capped non-displaceable subsurface. Thus, we need to check that this group has no non-trivial finite normal subgroups. We know that $\operatorname{Map}(\widehat K)$ satisfies this, but that is not enough to have that $\operatorname{PMap}(\widehat K)$ satisfies this. Fortunately, $\operatorname{PMap}(\widehat K)$ is torsion-free if $\widehat K$ has enough punctures. Thus, up to enlarging $K$ again, $\operatorname{PMap}(\widehat K)$ is torsion-free and thus has no non-trivial finite normal subgroups}. 
The conclusion follows from Proposition~\ref{prop:subgroup-criterion}.
\end{example}

\begin{figure}
    \centering
    \includegraphics[width=0.5\linewidth]{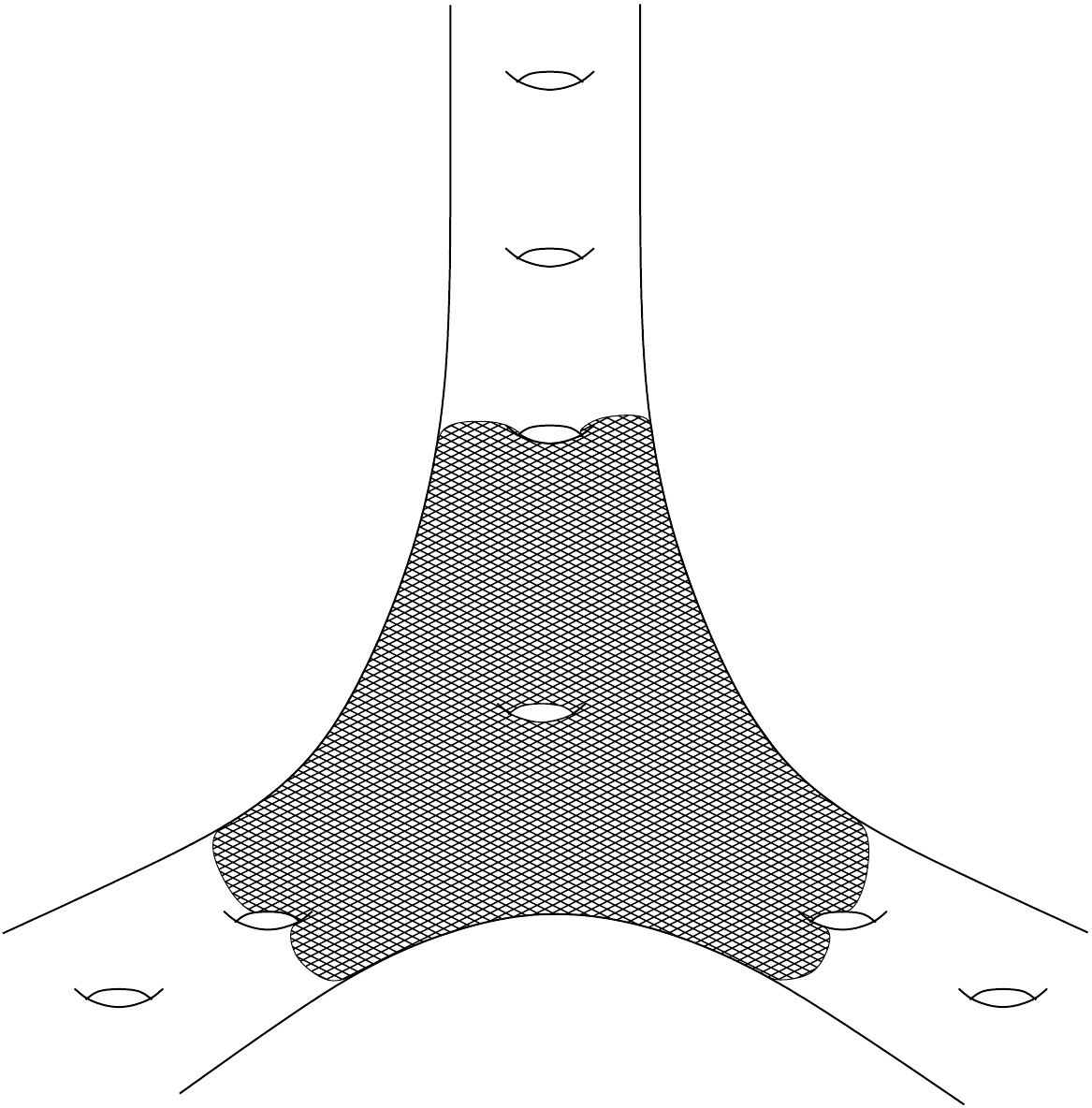}
    \caption{An infinite-type surface $S$ with a shaded non-displaceable compact subsurface.}
    \label{fig:fig1}
\end{figure}

%\begin{example}[Finite-index subgroups]
%\label{ex:finite-index}
%Let \(S\) be a connected orientable infinite-type surface containing a connected finite-type non-displaceable subsurface. Then every finite-index subgroup \( H<\operatorname{Map}(S) \) has property \(P_{\mathrm{naive}}\).

%Indeed, the given subsurface is non-displaceable relative to \(H\). Moreover, every infinite cyclic subgroup of \(\operatorname{Map}(S)\) contains a finite-index subgroup lying in \(H\). Thus suitable positive powers of compactly supported independent pseudo-Anosov mapping classes and Dehn twists belong to \(H\). The corresponding finite-type quotient therefore acts non-elementarily and has no non-trivial finite normal subgroup. Proposition \ref{prop:subgroup-criterion} gives the conclusion. \end{example}

\begin{example}
Let $S$ be a connected orientable infinite-type surface with
$|\operatorname{Ends}(S)|\geq 3$. If
\(
\operatorname{PMap}_c(S)\leq H\leq \operatorname{PMap}(S),
\)
then $H$ has property $P_{\mathrm{naive}}$.

Indeed, by \cite[Proposition~1.2]{Hill}, $S$ contains a connected
finite-type subsurface which is non-displaceable relative to
$\operatorname{PMap}(S)$, and hence relative to $H$; see Figure \ref{fig:fig2} (top left) for an example. Moreover, since
$\operatorname{PMap}_c(S)\leq H$, the same argument as above and Proposition~\ref{prop:subgroup-criterion} imply the conclusion.
%for every sufficiently large
%finite-type enlargement $K$, the quotient $Q_H(K)$ contains
%$\operatorname{PMap}(\widehat K)$. Hence $Q_H(K)$ acts
%non-elementarily and acylindrically on $\mathcal C(\widehat K)$ and,
%for $\widehat K$ sufficiently complex, has no non-trivial finite normal
%subgroup. The conclusion follows from Proposition~\ref{prop:subgroup-criterion}.
\end{example}

\begin{example}[The once-punctured Loch Ness monster]
Let $S$ be the once-punctured Loch Ness monster surface. If
\(
\operatorname{PMap}_c(S)\leq H\leq \operatorname{PMap}(S),
\)
then $H$ has property $P_{\mathrm{naive}}$.

Indeed, a finite-type subsurface containing the isolated planar end is
non-displaceable relative to $\operatorname{PMap}(S)$, and hence relative
to $H$; see Figure \ref{fig:fig2} (top right) for an example. Since $\operatorname{PMap}_c(S)\leq H$, the finite-type quotient
condition in Proposition~\ref{prop:subgroup-criterion} holds after taking a sufficiently
large enlargement. Thus $H$ has property $P_{\mathrm{naive}}$.
\end{example}

\begin{example}
    Let $S$ be such that it has more than one end. If $\operatorname{PMap}_c(S)\leq H\leq \overline{\operatorname{PMap}_c(S)}$, then $H$ has property $P_{\mathrm{naive}}$.
    
    Indeed, since $H$ acts trivially on the ends and does not contain any handle-shift between different ends accumulated by genus, in this case there exists a finite-type subsurface $K$ that is non-displaceable relative to $H$; see Figure \ref{fig:fig2} (bottom left and bottom right) for an example. By the same argument as in the previous examples, the conclusion follows.
\end{example}

\begin{figure}
    \centering
    \includegraphics[width=0.4\linewidth]{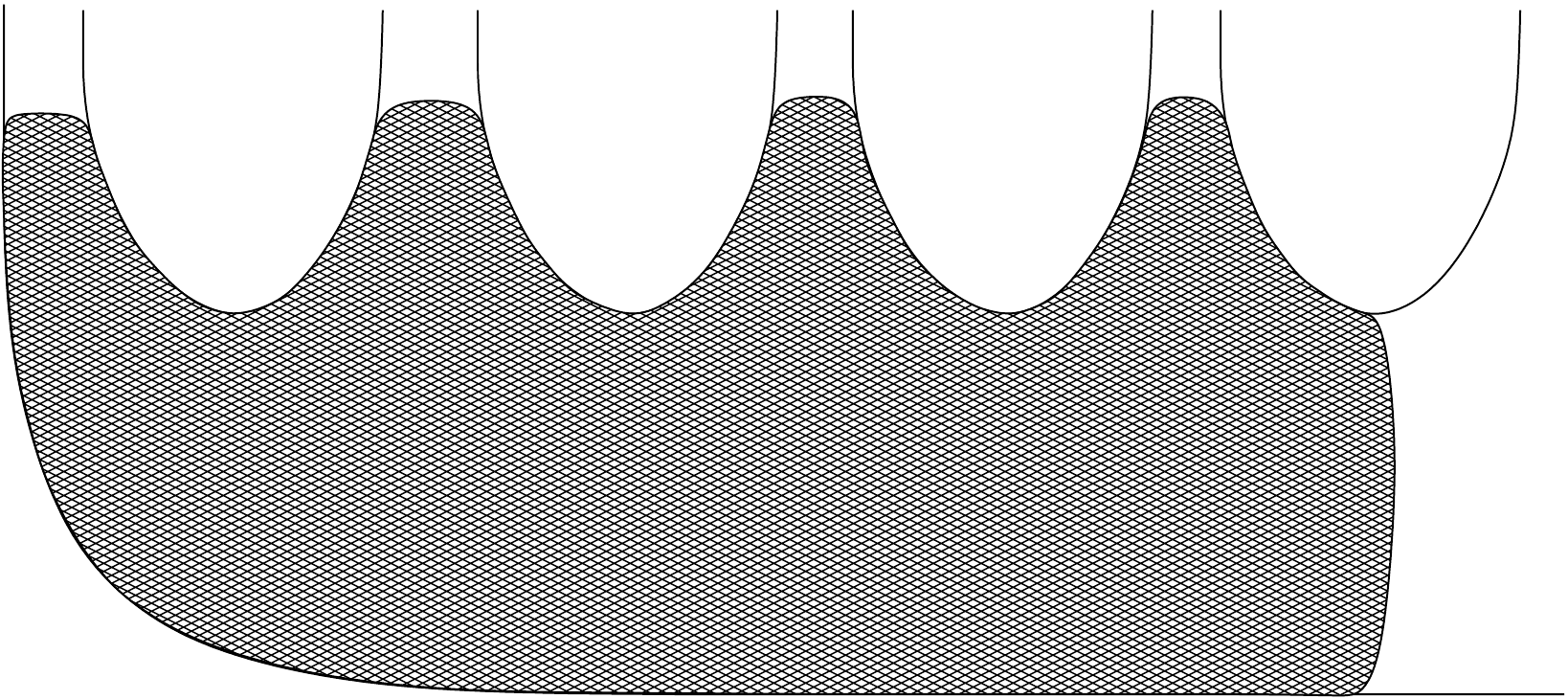} \hspace{5mm}
    \includegraphics[width=0.4\linewidth]{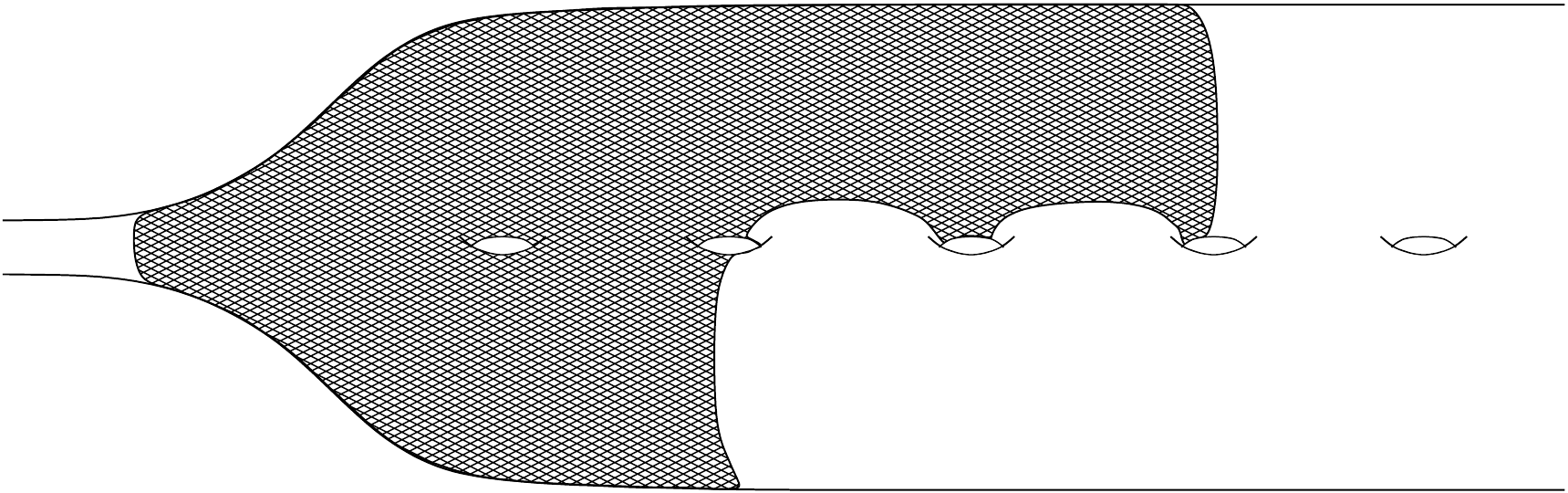} \\[3mm]    
    \includegraphics[width=0.4\linewidth]{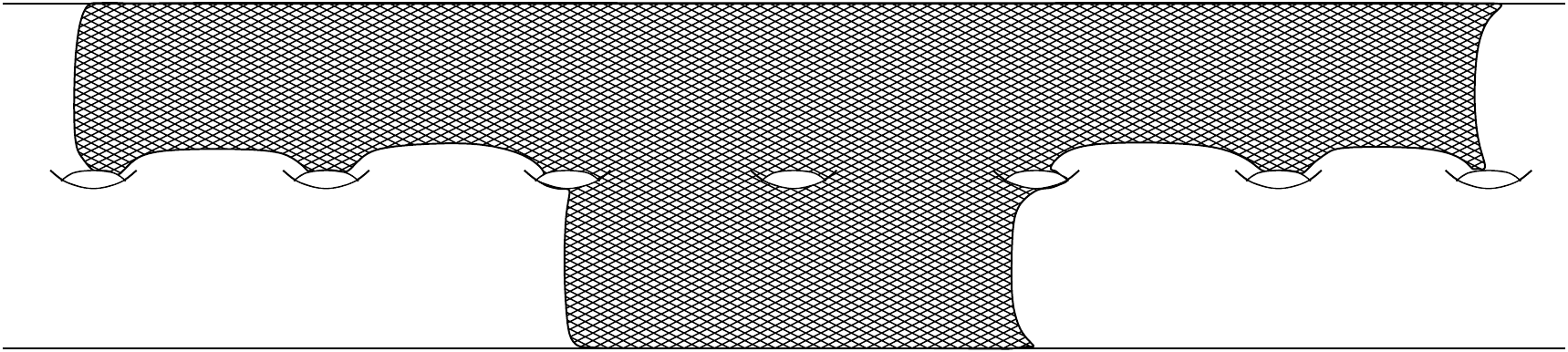} \hspace{5mm}
    \includegraphics[width=0.4\linewidth]{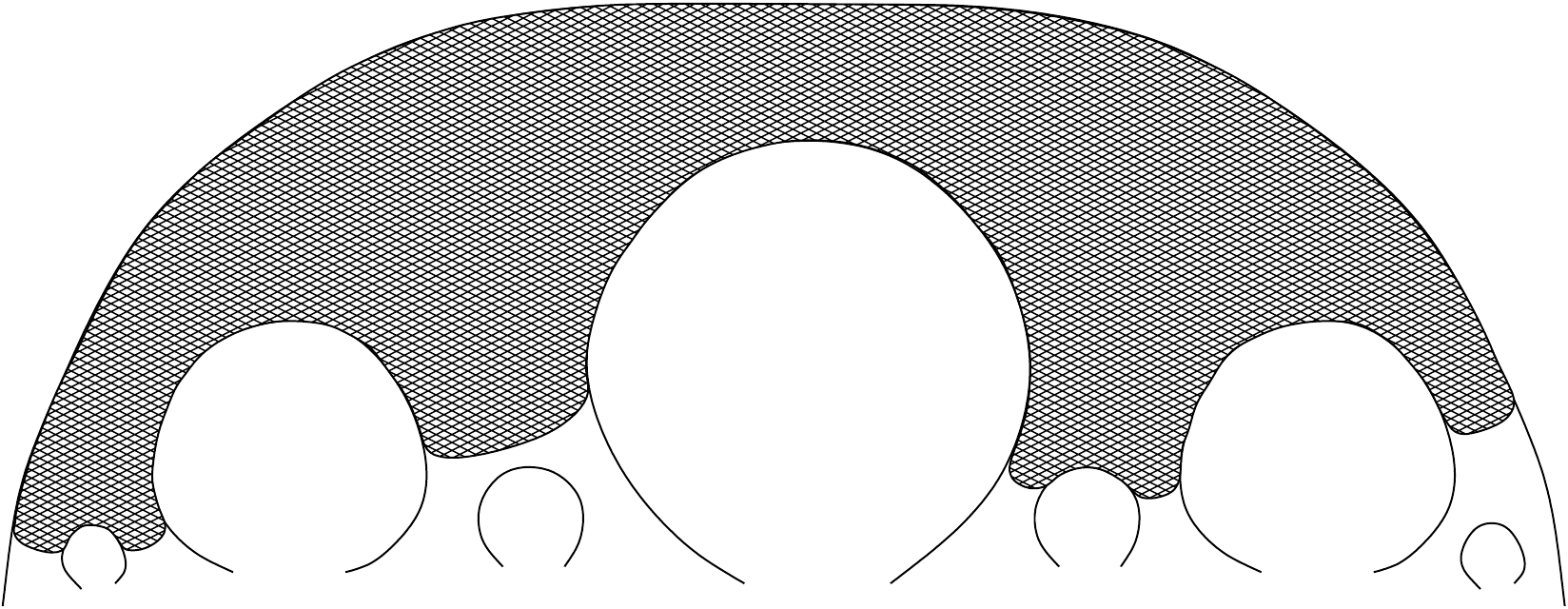}
    \caption{Infinite-type subsurface with shaded non-displaceable compact subsurface relative to the subgroups $\operatorname{PMap}_c(S) \leq H \leq \overline{\operatorname{PMap}_c(S)}$.}
    \label{fig:fig2}
\end{figure}

The subgroup criterion does not apply to arbitrary subgroups. For example,
if \(H\) is generated by commuting Dehn twists supported in a finite type
subsurface, then \(H\) may be abelian. Such a subgroup cannot have
\(P_{\text{naive}}\). Thus the existence of a non-displaceable subsurface
for the ambient group does not by itself imply \(P_{\text{naive}}\) for
every subgroup.

\section{The non-orientable case}\label{sec:nonorientable}

Let $N$ be a connected non-orientable surface of infinite type.  We now pass
to its orientation double cover and apply
Proposition \ref{prop:subgroup-criterion} to the lifted subgroup.

\subsection{The orientation double cover and the lift homomorphism}

The orientation double cover is the two-sheeted cover
\(
p:\widetilde N\rightarrow N
\)
defined by
\[
 \widetilde N
 =\{(x,o_x):x\in N,\ o_x\text{ is a local orientation at }x\}, ~p(x,o_x)=x.
\]
Its deck involution is
\(
       \tau(x,o_x)=(x,-o_x).
\)
The surface $\widetilde N$ is orientable, and if $N$ is of infinite type then
so is $\widetilde N$. To emphasize the orientability of $\widetilde N$, we now denote the orientation double cover of $N$ simply by $S$, i.e. $S := \widetilde N$.

\begin{lemma}\label{lem:canonical-lift}
Let $N$ be non-orientable and let $p:S\to N$ be its orientation
double cover.  Every mapping class $f\in\operatorname{Map}(N)$ has a canonical
orientation-preserving lift $\widetilde f\in\operatorname{Map}(S)$.  This gives a
homomorphism
\[
       \iota:\operatorname{Map}(N)\longrightarrow\operatorname{Map}(S),
       ~ f\longmapsto\widetilde f,
\]
and
\(\widetilde f\,\tau=\tau\,\widetilde f
\)
in the extended mapping class group of $S$.  In particular, if $N$ is of
infinite type, then $\iota$ is injective.
\end{lemma}

\begin{proof}
For a representative homeomorphism $f:N\to N$, define
\[
       \widetilde f(x,o_x)=(f(x),f_*(o_x)).
\]
This is the unique orientation-preserving lift of $f$ and is compatible with
composition.  It also satisfies
\[
 \widetilde f\tau(x,o_x)
 =(f(x),-f_*(o_x))
 =\tau\widetilde f(x,o_x).
\]
Thus the construction descends to a homomorphism on mapping class groups.
Injectivity in the infinite-type case is
\cite[Corollary 5.4]{CHR26}; see also the summary statement
\cite[Corollary 5]{CHR26}.
\end{proof}

Henceforth put
\(
G:=\iota(\operatorname{Map}(N))\leq\operatorname{Map}(S).
\)
The point of working with $G$, rather than with the whole group
$\operatorname{Map}(S)$, is that every element constructed in $G$ descends
automatically to $N$.

\begin{lemma}\label{lem:lift-K-connected}
Let $K\subset N$ be a connected non-orientable finite-type subsurface and
put $\widetilde K=p^{-1}(K)$.  Then $\widetilde K$ is connected and
$p|_{\widetilde K}:\widetilde K\to K$ is the orientation double cover of
$K$.
\end{lemma}

\begin{proof}
The restriction is the orientation double cover by construction.  If
$\widetilde K$ were disconnected, each component would map
homeomorphically onto $K$.  The orientation of either component would then
induce an orientation on $K$, contradicting non-orientability.
\end{proof}

\begin{lemma}\label{lem:nondisp-lifts}
Let \(H<\operatorname{Map}(N)\) and put
\(\widetilde H=\iota(H)\). Suppose that
\(K\subset N\) is a connected essential non-orientable finite-type
subsurface with \(g+b>5\). If \(K\) is non-displaceable
relative to \(H\), then
\(\widetilde K=p^{-1}(K)\) is non-displaceable relative to
\(\widetilde H\).
\end{lemma}

\begin{proof}
Let \(\Phi\in\operatorname{Homeo}^{+}(S)\) be an arbitrary
representative of an element of \(\widetilde H\). Choose \(h\in H\),
a representative \(f\in\operatorname{Homeo}(N)\) of \(h\), and let
\(\widetilde f\) be its canonical lift. Then
\(
[\Phi]=\iota(h)=[\widetilde f].
\)

Since \(K\) is non-exceptional, there exist two-sided essential curves \(a,b\subset\operatorname{int}(K)\) which fill \(K\)
\cite[Lemma~2.1]{EGPS23}. For any such two two-sided essential curves
\(a,b\subset N\),
\[
i_S\bigl(p^{-1}(a),p^{-1}(b)\bigr)
   =2\,i_N(a,b).
\]

Suppose, towards a contradiction, that
\(\Phi(\widetilde K)\cap\widetilde K=\varnothing\). For every
\(c,d\in\{a,b\}\), the multicurves \(p^{-1}(c)\) and
\(\Phi(p^{-1}(d))\) are disjoint. Since \(\Phi\) is isotopic to
\(\widetilde f\), it follows that
\[
0
=i_S\bigl(p^{-1}(c),p^{-1}(f(d))\bigr)
=2\,i_N(c,f(d)).
\]
Thus every curve in \(\{a,b\}\) has zero geometric intersection with
every curve in \(\{f(a),f(b)\}\). After placing these curves in minimal position, the two filling pairs are
disjoint. Since they fill \(K\) and \(f(K)\), respectively, the
subsurfaces \(K\) and \(f(K)\) admit disjoint representatives.

By isotopy extension, there is a representative \(f_0\) of \(h\)
such that
\(
f_0(K)\cap K=\varnothing,
\)
contradicting the non-displaceability of \(K\) relative to \(H\).
Therefore every representative of every element of \(\widetilde H\)
meets \(\widetilde K\), as required.
\end{proof}

The next lemma is the point at which the enlargement mechanism from the
orientable proof is adapted to the orientation double cover.

\begin{lemma}\label{lem:tau-powerdetect}
Let $N$ be a connected non-orientable surface of infinite type,
let $p:S\to N$ be its orientation double cover, and let
$H\leq\operatorname{Map}(N)$. Put $\widetilde H=\iota(H)$.
Suppose that $K_0\subset N$ is a connected non-orientable finite-type
subsurface which is non-displaceable relative to $H$.

For every finite set
$F=\{f_1,\ldots,f_n\}\subset H\setminus\{1\}$,
there exists a connected essential non-orientable finite-type
subsurface $K\supset K_0$, still non-displaceable relative to $H$,
such that, with $\widetilde K=p^{-1}(K)$,
\[
\langle\iota(f_i)\rangle
\cap\operatorname{Fix}_{\widetilde H}(\widetilde K)
=\{1\}
\qquad (1\leq i\leq n).
\]
Moreover, $\widetilde K$ is connected, essential and
$\tau$-invariant.
\end{lemma}

\begin{proof}
Put $\widetilde K_0=p^{-1}(K_0)$. By
Lemmas~\ref{lem:lift-K-connected} and~\ref{lem:nondisp-lifts},
the subsurface $\widetilde K_0$ is connected and non-displaceable
relative to $\widetilde H$. Since $\iota$ is injective, we may apply
Lemma~\ref{power-detect} to $\widetilde H$, $\widetilde K_0$ and
$\iota(F)$. We obtain a connected essential finite-type subsurface
$L\supset\widetilde K_0$ such that
\[
\langle\iota(f_i)\rangle
\cap\operatorname{Fix}_{\widetilde H}(L)=\{1\}
\qquad (1\leq i\leq n).
\]

Take a connected $\tau$-invariant finite-type regular neighborhood
of $L\cup\tau(L)$ and project it to $N$. Enlarge its image, if
necessary, to a connected essential non-orientable finite-type
subsurface $K$ containing $K_0$, and put
$\widetilde K=p^{-1}(K)$. Then $\widetilde K$ is connected,
essential and $\tau$-invariant, and it contains $L$. Since $K$
contains $K_0$, it remains non-displaceable relative to $H$.

Finally,
\(
\operatorname{Fix}_{\widetilde H}(\widetilde K)
\subseteq\operatorname{Fix}_{\widetilde H}(L).
\)
Hence the required power-detection condition follows immediately
from the corresponding condition for $L$.
\end{proof}

\subsection{The lifted subgroup criterion}

\begin{theorem}[Non-orientable subgroup criterion]
\label{thm:nonorientable-subgroup}
Let $N$ be a connected non-orientable infinite-type surface,
$p:S\to N$ its orientation double cover, and $H<\operatorname{Map}(N)$.  Put
$\widetilde H=\iota(H)$.
Assume that for every finite set $F\subset H\setminus\{1\}$ there exists a
connected essential non-orientable finite-type subsurface $K\subset N$ such
that, with $\widetilde K=p^{-1}(K)$,
\begin{enumerate}
\item $K$ is non-displaceable relative to $H$;
\item for every $h\in F$,
\(
\langle\iota(h)\rangle
       \cap\operatorname{Fix}_{\widetilde H}(\widetilde K)=\{1\};
\)
\item the group
\(
Q_{\widetilde H}(\widetilde K)
       :=\rho_{\widetilde K}
       \bigl(\operatorname{Stab}_{\widetilde H}(\widetilde K)\bigr)
\)
acts non-elementarily on $\mathcal{C}(\widehat{\widetilde K})$ and has no
non-trivial finite normal subgroup.
\end{enumerate}
Then $H$ has property $P_{\text{naive}}$.
\end{theorem}

\begin{proof}
By Lemma~\ref{lem:nondisp-lifts}, $\widetilde K$ is non-displaceable relative
to $\widetilde H$.  Conditions (2) and (3) are precisely the remaining
hypotheses of Proposition~\ref{prop:subgroup-criterion}, applied to the orientable
surface $\widetilde S$, the subgroup $\widetilde H$ and the finite set
$\iota(F)$.  Hence there is an infinite-order element
$\widetilde g\in\widetilde H$ such that
\[
       \langle\widetilde g,\iota(h)\rangle
       \cong
       \langle\widetilde g\rangle*\langle\iota(h)\rangle
       ~(h\in F).
\]

Since $\widetilde g\in\widetilde H=\iota(H)$, there exists $g\in H$ with
$\widetilde g=\iota(g)$.  The element $g$ has infinite order because $\iota$
is injective.  If a non-trivial reduced word in $g$ and $h$ represented the
identity in $H$, applying $\iota$ would give the corresponding non-trivial
reduced relation between $\widetilde g$ and $\iota(h)$, a contradiction.
Thus
\[
       \langle g,h\rangle
       \cong\langle g\rangle*\langle h\rangle
       ~(h\in F).
\]
Since $F$ was arbitrary, $H$ has property $P_{\text{naive}}$.
\end{proof}

The descent argument used above can also be isolated as follows.

\begin{lemma}\label{lem:free-descend}
Let $f,g\in\operatorname{Map}(N)$.  If
\[
       \langle\iota(f),\iota(g)\rangle
       \cong
       \langle\iota(f)\rangle*\langle\iota(g)\rangle
\]
in $\operatorname{Map}(S)$, then
\[
       \langle f,g\rangle
       \cong\langle f\rangle*\langle g\rangle
\]
in $\operatorname{Map}(N)$.
\end{lemma}

\begin{proof}
A non-trivial reduced relation in $f$ and $g$ would map under the injective
homomorphism $\iota$ to the same reduced relation in $\iota(f)$ and
$\iota(g)$, contradicting the assumed free-product decomposition.
\end{proof}

Next we give the proof of Theorem~\ref{thm:main}.

\begin{proof}[Proof of Theorem~\ref{thm:main}]

Let $N_0$ be a compact non-displaceable subsurface of $N$. Since $N$ is of infinite type, we can always enlarge it keeping it compact so that $N_0$ is a compact connected non-orientable subsurface of genus $g$ and $b$ boundary components, such that $g + b > 5$ and $b \geq 2g$. Note that under these circumstances $\operatorname{Map}(N_0)$ admits a pseudo-Anosov element, and the existence of pseudo-Anosov elements is going to persist in any enlargement of $N_0$.

Let $K \supset N_0$ be a compact connected essential non-orientable subsurface. Then, being $G := \iota(\operatorname{Map}(N))$, $\operatorname{Stab}_G(\widetilde K)$ admits a $K$-pseudo-Anosov, which implies that $Q_G(\widetilde K) \leq  \operatorname{PMap}(\widehat{\widetilde K})$ has a pseudo-Anosov map. Due to the conditions on $N_{0}$, $\operatorname{Map}(\widehat{\widetilde K})$ acts acylindrically on $\mathcal{C}(\widehat{\widetilde K})$. Then, $Q_G(\widetilde K)$ acts acylindrically on $\mathcal{C}(\widehat{\widetilde K})$. Moreover, \(Q_G(\widetilde K)\) is not virtually cyclic. Indeed, choose an essential two-sided curve \(c\subset\operatorname{int}(K)\) such that, writing \(p^{-1}(c)=\alpha\sqcup\tau(\alpha)\), the curves \(\alpha\) and \(\tau(\alpha)\) remain essential and non-isotopic in \(\widehat{\widetilde K}\). Extending \(T_c\) by the identity outside \(K\), its canonical lift induces
$$
\rho_{\widetilde K}(\iota(T_c))
   =T_\alpha T_{\tau(\alpha)}^{-1}
   \in Q_G(\widetilde K),
$$
up to replacing \(T_\alpha T_{\tau(\alpha)}^{-1}\) by its inverse. The element \(T_\alpha T_{\tau(\alpha)}^{-1}\) has infinite order and is elliptic on \(\mathcal C(\widehat{\widetilde K})\), since it fixes \([\alpha]\). If \(Q_G(\widetilde K)\) were virtually cyclic, then \(T_\alpha T_{\tau(\alpha)}^{-1}\) and the pseudo-Anosov element above would have non-zero powers that coincide, which is impossible because non-zero powers of the former are elliptic whereas those of the latter are loxodromic. Therefore, the classification of acylindrical actions by Osin \cite{O16} implies that $Q_G(\widehat{\widetilde K})$ acts also non-elementarily on $\mathcal{C}(\widehat{\widetilde K})$.

Finally, for $K$ chosen so that $g+b>5$, the capped
orientation double cover identifies $Q_G(\widetilde K)$ with a subgroup
of $\operatorname{Map}(\widehat K)$ containing all Dehn twists about
generic two-sided curves. Every finite normal subgroup therefore
centralizes a non-zero power of each such twist, and hence the entire
twist subgroup by \cite[Corollary~4.5 and Proposition~4.6]{Stu06}.
Since $\widehat K$ has no boundary, the centralizer of this twist
subgroup is trivial by \cite[Theorem~6.2]{Stu06}. Thus
$Q_G(\widetilde K)$ has no non-trivial finite normal subgroup.
Therefore, $N$ satisfies the symmetric finite-type quotient condition.

Let
\(F=\{f_1,\ldots,f_n\}
       \subset\operatorname{Map}(N)\setminus\{1\}.
\)
Let $K_0$ be the non-displaceable core supplied by the symmetric finite-type quotient condition.  Apply Lemma~\ref{lem:tau-powerdetect} to obtain a
connected essential non-orientable finite-type enlargement $K_1\supset K_0$
for which
\[
\langle\iota(f_i)\rangle
       \cap\operatorname{Fix}_{G}(\widetilde K_1)=\{1\}
       ~(1\leq i\leq n),
\]
with $\widetilde{K}_1 := p^{-1}(K_1)$.
By the symmetric finite-type quotient condition, enlarge $K_1$ further to a connected essential
non-orientable finite-type subsurface $K\supset K_1$ such that
\(Q_{G}(\widetilde K)\)
acts non-elementarily on $\mathcal{C}(\widehat{\widetilde K})$ and has no
non-trivial finite normal subgroup, as in the beginning of the proof, again with $\widetilde K := p^{-1}(K)$.  Since
\(       \operatorname{Fix}_{G}(\widetilde K)
       \subseteq \operatorname{Fix}_{G}(\widetilde K_1),
\)
the power-detection condition persists:
\[
       \langle\iota(f_i)\rangle
       \cap\operatorname{Fix}_{G}(\widetilde K)=\{1\}
       ~(1\leq i\leq n).
\]
Moreover, $K$ contains the non-displaceable core $K_0$, hence is itself
non-displaceable.  Thus all hypotheses of
Theorem~\ref{thm:nonorientable-subgroup} hold for $H=\operatorname{Map}(N)$.  Therefore
$\operatorname{Map}(N)$ has property $P_\text{naive}$.
\end{proof}

\begin{remark}[Checking the finite-type hypothesis]\label{rem:verifySFQ}
The symmetric finite-type quotient condition separates the infinite-type
argument from a finite-type verification.  A convenient sufficient route is
the following.  For every sufficiently large non-orientable finite-type
enlargement $K$ of the non-displaceable core, verify that the symmetric
restriction group
$Q_{G}(\widetilde K)$ contains two independent pseudo-Anosov elements.
Its action on $\mathcal{C}(\widehat{\widetilde K})$ is then non-elementary, while
acylindricity is inherited from the Bowditch action of
$\operatorname{Map}(\widehat{\widetilde K})$.  It remains to rule out a finite normal
subgroup.  In the usual high-complexity non-orientable range this can be
done using powers of Dehn twists about two-sided curves together with
rigidity of the two-sided curve complex; compare Atalan-Korkmaz
\cite{AK14} and Irmak-Paris \cite{IP20}.  This verification is finite type
and is logically independent of the power-detection and descent arguments
proved above.
\end{remark}

%\begin{remark}[Why one must work inside the lifted subgroup]
%\label{rem:inside-lifted}
%The statement that $\operatorname{Map}(S)$ has property $P_\text{naive}$ is not by itself sufficient to prove property $P_\text{naive}$ for $\operatorname{Map}(N)$.  An arbitrary witness in $\operatorname{Map}(S)$ need not descend to $N$.  In the proof above the orientable subgroup criterion is applied directly to $G=\iota(\operatorname{Map}(N))$. The common pseudo-Anosov witness therefore belongs to $G$ from the outset and is automatically the canonical lift of a mapping class of $N$. No product construction and no separate argument asserting preservation of WWPD under such a product is needed. \end{remark}

\bibliographystyle{alpha}
\bibliography{sample}

\end{document}